\documentclass[reqno,11pt]{amsart}

\usepackage[all]{xy}
\usepackage{amssymb}
\usepackage{amsgen}
\usepackage{amsmath}
\usepackage{amsthm}
\usepackage{mathrsfs}
\usepackage{cite}
\usepackage{amsfonts}

\newcommand{\Bis}{\Gamma}
\newcommand{\ann}{\mathop{\mathrm{Ann}}}
\newcommand{\skel}[1]{^{(#1)}}

\newcommand{\dom}{\mathop{\boldsymbol d}\nolimits}
\newcommand{\ran}{\mathop{\boldsymbol r}\nolimits}

\renewcommand{\to}{\longrightarrow}

\newcommand{\supp}{\mathop{\mathrm{supp}}}

\newcommand{\modu}[1]{#1\text{-}\mathrm{mod}}

\newcommand{\Ind}{\mathop{\mathrm{Ind}}\nolimits}

\newcommand{\inv}{^{-1}}

\usepackage{xcolor}

\newtheorem{Thm}{Theorem}
\newtheorem{Prop}[Thm]{Proposition}

{\theoremstyle{definition}
}
{\theoremstyle{remark}
\newtheorem{Rmk}[Thm]{Remark}}
\newtheorem{Cor}[Thm]{Corollary}
{\theoremstyle{remark}
}
{\theoremstyle{remark}
}

{\theoremstyle{remark}
}
{\theoremstyle{remark}
}
{\theoremstyle{remark}
}

\title[Ideals of \'etale groupoid algebras]{Ideals of \'etale groupoid algebras and Exel's Effros-Hahn conjecture}

\author{Benjamin Steinberg}
\address[B.~Steinberg]{%
    Department of Mathematics\\
    City College of New York\\
    Convent Avenue at 138th Street\\
    New York, New York 10031\\
    USA}
\email{bsteinberg@ccny.cuny.edu}

\thanks{This work was supported in part by a PSC-CUNY grant and by the Fulbright Commission, which supported the author's visit to Brazil where much of this work was done.  He thanks the Universidade Federal de Santa Catarina for its hospitality}
\date{\today}


\begin{document}

\begin{abstract}
We extend to arbitrary commutative base rings a recent result of Demeneghi that every ideal of an ample groupoid algebra over a field is an intersection of kernels of induced representations from isotropy groups, with a much shorter proof, by using the author's Disintegration Theorem for groupoid representations.  We also prove that every primitive ideal is the kernel of an induced representation from an isotropy group; however,  we are unable to show, in general, that it is the kernel of an irreducible induced representation.  If each isotropy group is finite (e.g., if the groupoid is principal) and if the base ring is Artinian (e.g., a field), then we can show that every primitive ideal is the kernel of an irreducible representation induced from isotropy. 
\end{abstract}

\maketitle

\section{Introduction}
The original Effros-Hahn conjecture~\cite{EffHahn,EffHahnBull} suggested that every primitive ideal of a crossed product of an amenable locally compact group with a commutative $C^*$-algebra should be induced from a primitive ideal of an isotropy group.  The result was proved by Sauvageot~\cite{Sauvage} for discrete groups and a more general result than the original conjecture was proved by Gootman and Rosenberg in~\cite{Effhahnpf}.  Crossed products of the above form are special cases of groupoid $C^*$-algebras and analogues of the Effros-Hahn conjecture in the groupoid setting were achieved by Renault~\cite{RenaultEH} and Ionescu and Williams~\cite{IonescuWilliams}.

In~\cite{mygroupoidalgebra}, the author initiated the study of convolution algebras of ample groupoids over commutative rings with unit; see also~\cite{operatorguys1}.  R.~Exel conjectured at the PARS meeting in Gramado, 2014 (and perhaps earlier) that an analogue of the Effros-Hahn conjecture should hold in this context.  The author had developed in~\cite{mygroupoidalgebra} a theory of induction from isotropy groups in this setting and had proven that inducing an irreducible representation from an isotropy group results in an irreducible representation of the groupoid algebra.

In~\cite{ExelDoku}, Dokuchaev and Exel showed that if a discrete group  $G$ acts partially on a locally compact and totally disconnected space $X$, then every ideal of the partial crossed product $C_c(X,\Bbbk)\rtimes G$, where $C_c(X,\Bbbk)$ is the ring of locally constant, compactly supported functions from $X$ to the field $\Bbbk$, is an intersection of ideals induced from isotropy.  Note that such partial crossed products are ample groupoid convolution algebras.  Since in a $C^*$-algebra, every closed ideal is an intersection of primitive ideals, this result can be viewed as an analogue of Effros-Hahn for partial crossed products.

Demeneghi~\cite{demeneghi} extended the result of Dokuchaev and Exel to arbitrary ample groupoid algebras over a field.  Namely, he showed that each ideal is an intersection of kernels of induced representations from isotropy subgroups.  His proof is rather indirect.  First he develops a theory of induced representations for crossed products of the form $C_c(X,\Bbbk)\rtimes S$ where $S$ is an inverse semigroup acting on a locally compact and totally disconnected space $X$.  Then he proves the result for such crossed products.  Finally, he proves that groupoid convolution algebras are such crossed products using the full strength of his theory (and the converse is essentially true as well) and he shows that crossed product induction corresponds to groupoid induction under the isomorphism.  His paper is around 50 pages in all.

In this paper, we prove that over an arbitrary base commutative ring each ideal of an ample groupoid convolution algebra is an intersection of kernels of induced representations from isotropy groups. Moreover, our proof is direct --- circumventing entirely the crossed product machinery --- and short.  It relies on the author's Disintegration Theorem~\cite{groupoidbundles}, which shows that modules for ample groupoid convolution algebras come from sheaves on the groupoid.  This machinery is not very cumbersome to develop and is quite useful for analyzing irreducible representations, as was done in~\cite{groupoidprimitive}.  In future work, it will be shown that the Disintegration Theorem can be used to establish the isomorphism between inverse semigroup crossed products and groupoid algebras directly, without using induced representations.

We also obtain some new progress on Exel's original conjecture on the structure of primitive ideals for groupoid algebras.  Namely, we show that every primitive ideal is the kernel of a single representation induced from an isotropy group (rather than an infinite intersection of such kernels).  We are, unfortunately, not able to show in general that it is the kernel of an irreducible representation induced from an isotropy group.  We are, however, able to prove Exel's version of the Effros-Hahn conjecture on primitive ideals if the base ring $R$ is Artinian and each isotropy group is either finite, or locally finite abelian with orders of elements invertible in $R/J(R)$ where $J(R)$ is the Jacobson radical of $R$ (e.g., if $R$ has characteristic zero).

\section{Preliminaries}
This section summarizes definitions and results from~\cite{mygroupoidalgebra} and~\cite{groupoidbundles} that we use throughout.  There are no new results in this section.

\subsection{Groupoids}
Following Bourbaki, compactness will include the Hausdorff axiom throughout this paper.  However, we do not require locally compact spaces to be Hausdorff. A topological groupoid $\mathscr G=(\mathscr G\skel 0,\mathscr G\skel 1)$ is \emph{\'etale} if its domain map $\dom$ (or, equivalently, its range map $\ran$) is a local homeomorphism.  In this case, identifying objects with identity arrows, we have that $\mathscr G\skel 0$ is an open subspace of $\mathscr G\skel 1$ and the multiplication map is a local homeomorphism.  See, for example,~\cite{Paterson,resendeetale,Exel}.

Following~\cite{Paterson}, an \'etale groupoid is called \emph{ample} if its unit space $\mathscr G\skel 0$ is locally compact Hausdorff with a basis of compact open subsets. We shall say that an ample groupoid $\mathscr G$ is Hausdorff if $\mathscr G\skel 1$ is Hausdorff.

A \emph{local bisection} of an \'etale groupoid $\mathscr G$ is an open subset $U\subseteq \mathscr G\skel 1$ such that both $\dom|_U$ and $\ran|_U$ are homeomorphisms with their images.  The local bisections form a basis for the topology on $\mathscr G\skel 1$~\cite{Exel}. The set $\Bis(\mathscr G)$ of local bisections is an inverse monoid (cf.~\cite{Lawson}) under the binary operation \[UV = \{\gamma\eta\mid \gamma\in U,\eta\in V,\ \dom (\gamma)=\ran (\eta)\}.\] The semigroup inverse is given by $U\inv = \{\gamma\inv\mid \gamma\in U\}$.  The set $\Bis_c(\mathscr G)$ of compact local bisections  is an inverse subsemigroup of $\Bis(\mathscr G)$~\cite{Paterson}. Note that $\mathscr G$ is ample if and only if $\Bis_c(\mathscr G)$ is a basis for the topology on $\mathscr G\skel 1$~\cite{Exel,Paterson}.

If $u\in \mathscr G\skel 0$, then the \emph{orbit} $\mathcal O_u$ of $u$ consists of all $v\in \mathscr G\skel 0$ such that there is an arrow $\gamma$ with $\dom(\gamma)=u$ and $\ran(\gamma)=v$.  The orbits form a partition of $\mathscr G\skel 0$.
A subset $X\subseteq \mathscr G\skel 0$ is \emph{invariant} if it is a union of orbits.

If $u\in \mathscr G\skel 0$, the \emph{isotropy group} of $\mathscr G$ at $u$ is
\[G_u=\{\gamma\in \mathscr G\skel 1\mid \dom(\gamma)=u=\ran(\gamma)\}.\]
Isotropy groups of elements in the same orbit are isomorphic.

\subsection{Ample groupoid algebras}
Fix a commutative ring with unit $R$.  The author~\cite{mygroupoidalgebra} associated an $R$-algebra $R\mathscr G$ to each ample groupoid $\mathscr G$ as follows.  We define $R\mathscr G$ to be the $R$-span in $R^{\mathscr G\skel 1}$ of the characteristic functions $1_U$ of compact open subsets $U$ of $\mathscr G\skel 1$.  It is shown in~\cite[Proposition~4.3]{mygroupoidalgebra} that $R\mathscr G$ is spanned by the elements $1_U$ with $U\in \Bis_c(\mathscr G)$.  If $\mathscr G\skel 1$ is Hausdorff, then $R\mathscr G$ consists of the locally constant $R$-valued functions on $\mathscr G\skel 1$ with compact support.  Convolution is defined on $R\mathscr G$ by \[f\ast g(\gamma)=\sum_{\dom(\eta)=\dom(\gamma)}f(\gamma \eta\inv)g(\eta)=\sum_{\alpha\beta=\gamma}f(\alpha)g(\beta).\]  The finiteness of the sums is proved in~\cite{mygroupoidalgebra}. The fact that the convolution belongs to $R\mathscr G$ comes from the computation $1_U\ast 1_V=1_{UV}$ for $U,V\in \Bis_c(\mathscr G)$~\cite{mygroupoidalgebra}.

The algebra $R\mathscr G$ is unital if and only if $\mathscr G\skel 0$ is compact, but it always has local units (i.e., is a directed union of unital subrings)~\cite{mygroupoidalgebra,groupoidbundles}.  A module $M$ over a ring $S$ with local units is termed \emph{unitary} if $SM=M$.  A module $M$ over a unital ring is unitary if and only if $1m=m$ for all $m\in M$.  The category of unitary $S$-modules is denoted $\modu{S}$.  Notice that every simple module is unitary; $M$ is simple if $SM\neq 0$ and $M$ has no proper, non-zero submodules.

\subsection{Induced modules}
We recall from~\cite{mygroupoidalgebra} the induction functor
\[\Ind_u\colon \modu{RG_u}\to \modu{R\mathscr G}.\]  For $u\in \mathcal G\skel 0$, let $\mathscr Gu = \dom\inv(u)$ denote the set of all arrows starting at $u$.  Then $G_u$ acts freely on the right of $\mathscr Gu$ by multiplication.  Hence $R\mathscr Gu$ is a free right $RG_u$-module.  A basis can be obtained by choosing, for each $v\in \mathcal O_u$, an arrow $\gamma_v\colon u\to v$.  We normally choose $\gamma_u=u$.  There is a left $R\mathscr G$-module structure on $R\mathscr Gu$ given by
\begin{equation}\label{eq:action}
f\alpha = \sum_{\dom(\gamma)=\ran(\alpha)}f(\gamma)\gamma\alpha
\end{equation}
for $\alpha\in \mathscr Gu$ and $f\in R\mathscr G$.  The $R\mathscr G$-action commutes with the $RG_u$-action by associativity and so $R\mathscr Gu$ is an $R\mathscr G$-$RG_u$-bimodule, unitary under both actions (as is easily checked).

 The functor $\Ind_u$ is defined by
\[\Ind_u(M) = R\mathscr Gu\otimes_{RG_u}M.\]
This functor is exact by freeness of $R\mathscr Gu$ as a right $RG_u$-module and there is, in fact, an $R$-module direct sum decomposition
\[\Ind_u(M) =\bigoplus_{v\in \mathcal O_u} \gamma_v\otimes M.\]  The action of $f\in R\mathscr G$ in these coordinates is given by
\begin{equation}\label{eq:action.ind}
f(\gamma_v\otimes m) =\sum_{w\in \mathcal O_u}\gamma_w\otimes \sum_{\gamma\colon v\to w}f(\gamma)(\gamma_w\inv \gamma\gamma_v)m,
\end{equation}
as is easily checked.  An immediate corollary of \eqref{eq:action.ind} is the following (see also~\cite{demeneghi}).

\begin{Prop}\label{p:annihilator.induced}
Let $u\in \mathscr G\skel 0$ and $M$  an $RG_u$-module.  Fix $\gamma_u\colon u\to v$ for all $v\in \mathcal O_u$. Then the equality
\[\ann(\Ind_u(M)) = \left\{f\in R\mathscr G\mid \forall v,w\in \mathcal O_u, \sum_{\gamma\colon v\to w} f(\gamma)(\gamma_w\inv \gamma\gamma_v)\in \ann(M)\right\}\] holds.
\end{Prop}

A crucial result is that induction preserves simplicity.

\begin{Thm}[Steinberg, Prop.~7.19, Prop.~7.20 of~\cite{mygroupoidalgebra}]\label{t:induced.simple}
Let $M$ be a simple $RG_u$-module with $u\in \mathscr G\skel 0$.  Then $\Ind_u(M)$ is a simple $R\mathscr G$-module.  Moreover, the functor $\Ind_u$ reflects isomorphism and $\Ind_u(M)\cong \Ind_v(N)$ implies $\mathcal O_u=\mathcal O_v$.
\end{Thm}

The  Effros-Hahn conjecture for ample groupoids, first stated to the best of our knowledge by Exel at the PARS2014 conference in Gramado (see also~\cite{ExelDoku}), says that each primitive ideal of $R\mathscr G$ is of the form $\ann(\Ind_u(M))$ for some $u\in \mathscr G\skel 0$ and simple $RG_u$-module $M$.

\subsection{Modules over ample groupoid algebras}
The key ingredient to our approach is the author's analogue~\cite{groupoidbundles} of Renault's Disintegration Theorem~\cite{renaultdisintegration} for modules over groupoid algebras. Here $R$ will be a fixed commutative ring with unit and $\mathscr G$ an ample groupoid (not necessarily Hausdorff).

A \emph{$\mathscr G$-sheaf} $\mathcal E=(E,p)$ consists of a space $E$, a local homeomorphism $p\colon E\to \mathscr G\skel 0$ and an action map $\mathscr G\skel 1\times_{\mathscr G\skel 0} E\to E$ (where the fiber product is with respect to $\dom$ and $p$), denoted $(\gamma,e)\mapsto \gamma e$, satisfying the following axioms:
\begin{itemize}
\item $p(e)e=e$ for all $e\in E$;
\item $p(\gamma e)=\ran(\gamma)$ whenever $p(e)=\dom(\gamma)$;
\item $\gamma(\eta e)=(\gamma\eta)e$ whenever $p(e)=\dom(\eta)$ and $\dom(\gamma)=\ran(\eta)$.
\end{itemize}

A \emph{$\mathscr G$-sheaf of $R$-modules} is a $\mathscr G$-sheaf $\mathcal E=(E,p)$ together with an $R$-module structure on each stalk $E_u=p\inv(u)$ such that:
\begin{itemize}
\item the zero section, $u\mapsto 0_u$ (the zero of $E_u$), is continuous;
\item addition $E\times_{\mathscr G\skel 0} E\to E$ is continuous;
\item scalar multiplication $R\times E\to E$ is continuous;
\item for each $\gamma\in \mathscr G\skel 1$, the map $E_{\dom(\gamma)}\to E_{\ran(\gamma)}$ given by $e\mapsto \gamma e$ is $R$-linear;
\end{itemize}
where $R$ has the discrete topology in the third item.
Note that the first three conditions are equivalent to $(E,p)$ being a sheaf of $R$-modules over $\mathscr G\skel 0$.
Crucial to this paper is that $E_u$ is an $RG_u$-module for each $u\in \mathscr G\skel 0$.

Note that the zero subspace \[\mathbf 0=\{0_u\mid u\in G\skel 0\}\] is an open subspace of $E$, being the image of a section of a local homeomorphism.

The \emph{support} of $\mathcal E$ is
\[\supp(\mathcal E) = \{u\in \mathscr G\skel 0\mid E_u\neq \{0_u\}\}.\]
Note that $\supp(\mathcal E)$ is an invariant subset of $\mathscr G_0$ but it need not be closed.


A \emph{(global) section} of $\mathcal E$ is a continuous mapping $s\colon \mathscr G\skel 0\to E$ such that $p\circ s=1_{\mathscr G\skel 0}$.  Note that if $s\colon \mathscr G\skel 0\to E$ is a section, then  its \emph{support}
\[\supp(s) = s\inv(E\setminus \mathbf 0)\] is closed.  We denote by $\Gamma_c(\mathcal E)$ the set of (global) sections with compact support.  Note that $\Gamma_c(\mathcal E)$ is an $R$-module with respect to pointwise operations and it be comes a unitary left $R\mathscr G$-module under the operation
\begin{equation}\label{eq:operation}
(fs)(u) =\sum_{\ran(\gamma)=u} f(\gamma)\gamma s(\dom(\gamma)) = \sum_{v\in \mathcal O_u}\sum_{\gamma\colon v\to u}f(\gamma)\gamma s(v).
\end{equation}
See~\cite{groupoidbundles} for details (where right actions and right modules are used).

If $e\in E_u$, then there is always a global section $s$ with compact support such that $s(u)=e$.  Indeed, we can choose a neighborhood $U$ of $e$ such that $p|_U\colon U\to p(U)$ is a homeomorphism with $p(U)$ open.  Then we can find a compact open neighborhood $V$  of $u$ with $V\subseteq U$ and define $s$ to be the restriction of $(p|_U)\inv$ on $V$ and $0$, elsewhere.  Then $s$ is continuous, $s(u)=e$ and the support of $s$ is closed and contained in $V$, whence compact.

Conversely, if $M$ is a unitary left $R\mathscr G$-module, we can define a $\mathscr G$-sheaf of $R$-modules $\mathrm{Sh}(M)= (E,p)$ where $E_u =\varinjlim_{u\in U}1_UM$ (with the direct limit over all compact open neighborhoods $U$ of $u$ in $\mathscr G\skel 0$) and if $\gamma\colon u\to v$ and $[m]_u$ is the class of $m$ at $u$, then $\gamma[m]_u = [Um]_v$ where $U$ is any compact local bisection containing $\gamma$. Here $E=\coprod_{u\in \mathscr G\skel 0} E_u$ has the germ topology.  See~\cite{groupoidbundles} for details.

\begin{Thm}[Steinberg~\cite{groupoidbundles}]\label{t:disint}
The functor $\mathcal E\mapsto \Gamma_c(\mathcal E)$ is an equivalence between the category $\mathcal B\mathscr G_R$ of $\mathscr G$-sheaves of $R$-modules and $\modu{R\mathscr G}$ with quasi-inverse $M\mapsto \mathrm{Sh}(M)$.
\end{Thm}

\begin{Rmk}\label{r:find.ideals}
As a consequence of Theorem~\ref{t:disint}, $M\cong \Gamma_c(\mathrm{Sh}(M))$ for any unitary module $M$ and so we have an equality of annihilator ideals \[\ann(M)=\ann(\Gamma_c(\mathrm{Sh}(M))).\]  In particular, we have $I=\ann(\Gamma_c(\mathrm{Sh}(R\mathscr G/I)))$.  Thus we can describe the ideal structure of $R\mathscr G$ in terms of the annihilators of modules of the form $\Gamma_c(\mathcal E)$.
\end{Rmk}

\section{The ideal structure of ample groupoid algebras}
In this section, we show how the Disintegration Theorem (Theorem~\ref{t:disint}) provides information about the ideal structure of $R\mathscr G$.

\subsection{General ideals}
Our main goal in this subsection is to prove the following theorem, generalizing a result of Demeneghi~\cite{demeneghi} that was originally proved over fields.  Demeneghi's proof is indirect, via crossed products, and therefore quite long.  Our proof is direct, using the Disintegration Theorem, and shorter, even including the 11 pages of~\cite{groupoidbundles}.

\begin{Thm}\label{t:annih}
Let $\mathscr G$ be an ample groupoid and $R$ a ring.  Let $\mathcal E=(E,p)$ be a $\mathscr G$-sheaf of $R$-modules.  Then the equality
\[\ann(\Gamma_c(\mathcal E))=\bigcap_{u\in \mathscr G\skel 0} \ann(\Ind_u(E_u))\] holds. Consequently, every ideal $I\lhd R\mathscr G$ is an intersection of annihilators of induced modules.
\end{Thm}
\begin{proof}
The final statement follows from the first by the equivalence of categories in Theorem~\ref{t:disint} (cf.~Remark~\ref{r:find.ideals}).  So we prove the first statement.
 Let $I=\ann(\Gamma_c(\mathcal E))$ and put $J_u = \ann(\Ind_u(E_u))$ for $u\in \mathscr G\skel 0$. Set $J=\bigcap_{u\in \mathscr G\skel 0}J_u$.   Then we want to prove that $I=J$.

Fix $u\in \mathscr G\skel 0$ and, for each $v\in \mathcal O_u$, choose $\gamma_v\colon u\to v$.  To show $I\subseteq J_u$, it suffices, by Proposition~\ref{p:annihilator.induced}, to show that if $f\in I$, then, for each $v,w\in \mathcal O_u$, we have that
\[\sum_{\gamma\colon v\to w}f(\gamma)(\gamma_w\inv \gamma\gamma_v)\in \ann(E_u).\]

So let $e\in E_u$ and let $s\in \Gamma_c(\mathcal E)$ with $s(u) =e$.  Since $\mathscr G$ is ample, the set $\ran\inv(w)\cap \supp(f)$ is finite. Since $\mathscr G\skel 0$ is Hausdorff, we can find $U\subseteq \mathscr G\skel 0$ compact open with $v\in U$ and $U\cap \dom(\ran\inv(w)\cap \supp(f))\subseteq \{v\}$.

Let $U_v$ be a compact local bisection containing $\gamma_v$ and $U_w$ a compact local bisection containing $\gamma_w$.  Replacing $U_v$ by $UU_v$, we may assume that
\[\ran(U_v)\cap \dom(\ran\inv (w)\cap \supp(f))\subseteq \{v\}.\]  By construction, the only elements in the support of $1_{U_w\inv}\ast f\ast 1_{U_v}$ with range $u$ are of the form $\gamma_w\inv\gamma\gamma_v$ with $\gamma\colon v\to w$ and $f(\gamma)\neq 0$.   As $1_{U_w\inv}\ast f\ast 1_{U_v}\in I$, we have by \eqref{eq:operation} that
\begin{align*}
0&=(1_{U_w\inv}\ast f\ast 1_{U_v}s)(u) \\ &= \sum_{\ran(\alpha)=u}(1_{U_w\inv}\ast f\ast 1_{U_v})(\alpha)\alpha s(\dom(\alpha)) \\ &=\sum_{\gamma\colon v\to w}f(\gamma)(\gamma_w\inv \gamma\gamma_v)s(u) \\ &=\sum_{\gamma\colon v\to w}f(\gamma)(\gamma_w\inv \gamma\gamma_v)e
\end{align*}
as required.  Thus $I\subseteq J_u$ for all $u\in \mathscr G\skel 0$.

Suppose now that $f\in J$ and let $s\in \Gamma_c(\mathcal E)$.  Then
\begin{equation}\label{eq:mustcheck}
(fs)(v) = \sum_{u\in \mathcal O_v}\sum_{\gamma\colon u\to v}f(\gamma)\gamma s(u).
\end{equation}
Let us fix $u\in \mathcal O_v$ and let $\gamma_u=u$ and $\gamma_v\colon u\to v$ be arbitrary.  Then, since $f\in J\subseteq J_u$, we have by Proposition~\ref{p:annihilator.induced} that
\[\sum_{\gamma\colon u\to v}f(\gamma)(\gamma_v\inv \gamma) s(u)=0.\] Multiplying on the left by $\gamma_v$ yields that, for each $u\in \mathcal O_v$,
\[\sum_{\gamma\colon u\to v}f(\gamma)\gamma s(u)=0\] and so the right hand side of \eqref{eq:mustcheck} is $0$.  Thus $f\in I$.  This completes the proof.
\end{proof}

\begin{Rmk}
  More concretely, if $I\lhd R\mathscr G$ is an ideal, then following the construction of the proof we see that  \[I = \bigcap_{u\in \mathscr G\skel 0}\ann\left(\Ind_u\left(\varinjlim_{u\in U}1_U(R\mathscr G/I)\right)\right)\] where $\varinjlim_{u\in U} 1_U(R\mathscr G/I)$ has the $RG_u$-module structure \[\gamma[1_Uf+I] = [1_V1_Uf+I]\] for $V$ any compact local bisection containing $\gamma$.
\end{Rmk}

\subsection{Primitive ideals}
Recall that an ideal $I$ of a ring is \emph{primitive} if it is the annihilator of a simple module. (Technically, we should talk about left primitive ideals since we are using left modules, but because groupoid algebras admit an involution, it doesn't matter.) We prove that each primitive ideal is the annihilator of a single induced representation (rather than an intersection of such annihilators, as in Theorem~\ref{t:annih}).  Unfortunately, we are not yet able to show, in general, that the induced representation is simple.   Still, this is new progress toward Exel's  Effros-Hahn conjecture.

\begin{Thm}\label{t:primitive.case}
Let $R$ be a commutative ring with unit and $\mathscr G$ an ample groupoid.  Let $I\lhd R\mathscr G$ be a primitive ideal.  Then $I=\ann(\Ind_u(M))$ for some $u\in \mathscr G\skel 0$ and $RG_u$-module $M$.
\end{Thm}
\begin{proof}
By Theorem~\ref{t:disint} we may assume that our simple module with annihilator $I$ is of the form $\Gamma_c(\mathcal E)$ for some $\mathscr G$-sheaf $\mathcal E=(E,p)$ of $R$-modules. Let $u\in \supp(\mathcal E)$  and put $J_u = \ann(\Ind_u(E_u))$.  We claim that $I=J_u$.
We know that $I\subseteq J_u$ by Theorem~\ref{t:annih};  we must prove the converse.  Suppose that $J_u$ does not annihilate $\Gamma_c(\mathcal E)$.  Then we can find a section $s$ with $J_us\neq \{0\}$.  As $J_u$ is an ideal, $J_us$ is a submodule and so $J_us=\Gamma_c(\mathcal E)$.  Let $0_u\neq e\in E_u$.  Then there is a section $t\in \Gamma_c(\mathcal E)$ such that $t(u)=e$.  Let $f\in J_u$ with $fs=t$.  Then we have that
\begin{equation}\label{eq:big.helper}
e=t(u)=(fs)(u) = \sum_{v\in \mathcal O_u}\sum_{\gamma\colon v\to u}f(\gamma)\gamma s(v).
\end{equation}

Let us fix $v\in \mathcal O_u$ and fix $\gamma_v\colon u\to v$; put $\gamma_u=u$.  Then by the assumption $f\in J_u$ and Proposition~\ref{p:annihilator.induced}, we have (since $\gamma_v\inv s(v)\in E_u$) that
\[0=\sum_{\gamma\colon v\to u} f(\gamma)(\gamma\gamma_v)(\gamma_v\inv s(v))=\sum_{\gamma\colon v\to u}f(\gamma)\gamma s(v).\]  Thus the right hand side of \eqref{eq:big.helper} is $0$, contradicting that $e\neq 0$.  We conclude that $J_u\subseteq I$.  This completes the proof.
\end{proof}

We now show that the $RG_u$-module $M$ above can be chosen to be simple under some strong hypotheses on the base ring and isotropy groups.  Let $J(S)$ denote the Jacobson radical of a ring $S$. A ring $S$ is called a \emph{left max ring} if each non-zero left $S$-module has a maximal (proper) submodule.  For example, any Artinian ring  $S$ is a left max ring.  Indeed, if $M\neq 0$, then $J(S)M\neq M$ by nilpotency of the Jacobson radical.  But $M/J(S)M$ is then a non-zero $S/J(S)$-module and every non-zero module over a semisimple ring is a direct sum of simple modules and hence has a simple quotient.  Thus $M$ has a maximal proper submodule.  A result of Hamsher~\cite{Hamsher} says that if $S$ is commutative, then $S$ is a left max ring if and only if $J(S)$ is $T$-nilpotent  (e.g., if $J(S)$ is nilpotent) and $S/J(S)$ is  von Neumann regular ring.   We now establish the  Effros-Hahn conjecture in the case that all isotropy group rings are left max rings.

\begin{Thm}\label{t:left.max}
Let $R$ be a commutative ring and $\mathscr G$ an ample groupoid such that $RG_u$ is a left max ring for all $u\in \mathscr G\skel 0$.  Then the primitive ideals of $R\mathscr G$ are exactly the ideals of the form $\ann(\Ind_u(M))$ where $M$ is a simple $RG_u$-module.
\end{Thm}
\begin{proof}
By Theorem~\ref{t:induced.simple} it suffices to show that any primitive ideal $I$ is of the form $\ann(\Ind_u(M))$ where $M$ is a simple $RG_u$-module.
By Theorem~\ref{t:disint} we may assume that our simple module with annihilator $I$ is of the form $\Gamma_c(\mathcal E)$ for some $\mathscr G$-sheaf $\mathcal E=(E,p)$ of $R$-modules. Let $u\in \supp(\mathcal E)$.  We already know from the proof of Theorem~\ref{t:primitive.case} that $I = \ann(\Ind_u(E_u))$.  Let $N$ be a maximal submodule of $E_u$ (which exists by assumption on $RG_u$) and let $J=\ann(\Ind_u(E_u/N))$.  Since $E_u/N$ is simple, it suffices to show that $J=I$.  Clearly, $I\subseteq J$ (by Proposition~\ref{p:annihilator.induced}) since $\ann(E_u)\subseteq \ann(E_u/N)$.  So it suffices to show that $J$ annihilates $\Gamma_c(\mathcal E)$.

Suppose that this is not the case.  Then there exists $s\in \Gamma_c(\mathcal E)$ with $Js\neq 0$.  Since $J$ is an ideal and $\Gamma_c(\mathcal E)$ is simple, we deduce $Js=\Gamma_c(\mathcal E)$.   Let $e\in E_u\setminus N$ (using that $N$ is a proper submodule) and let $t\in \Gamma_c(\mathcal E)$ with $t(u)=e$.  Then $t=fs$ with $f\in J$.  Let us compute
\begin{equation}\label{eq:big.helper.2}
e=t(u)=(fs)(u) = \sum_{v\in \mathcal O_u}\sum_{\gamma\colon v\to u}f(\gamma)\gamma s(v).
\end{equation}

Let us fix $v\in \mathcal O_u$, fix $\gamma_v\colon u\to v$ and set $\gamma_u=u$.  Then by the assumption $f\in J$ and Proposition~\ref{p:annihilator.induced}, we have (since $\gamma_v\inv s(v)\in E_u$) that
\[\sum_{\gamma\colon v\to u}f(\gamma)\gamma s(v)=\sum_{\gamma\colon v\to u} f(\gamma)(\gamma\gamma_v)(\gamma_v\inv s(v))\in N.\]  We deduce from \eqref{eq:big.helper.2} that $e\in N$, which is a contradiction. It follows that $J$ annihilates $\Gamma_c(\mathcal E)$ and so $I=J$.
\end{proof}

We obtain as a consequence a special case of Exel's Effros-Hahn conjecture, which will include the algebras of principal groupoids (or, more generally, groupoids with finite isotropy) over a field.  Recall that a groupoid is called \emph{principal} if all its isotropy groups are trivial.  

It is well known that a group ring $RG$ is Artinian if and only if $R$ is Artinian and $G$ is finite~\cite{PassmanBook}.  A result of Villamayor~\cite{villamayor} says that, for a group $G$ and commutative ring $R$, one has that $RG$ is von Neumann regular if and only if $R$ is von Neumann regular, $G$ is locally finite and the order of any element of $G$ is invertible in $R$.  The reader should recall that any Artinian semisimple ring is von Neumann regular and also that any von Neumann regular ring has a zero Jacobson radical~\cite[Cor.~4.24]{LamBook}.

\begin{Cor}\label{c:main}
Let $R$ be a commutative Artinian ring and $\mathscr G$ an ample groupoid such that each isotropy group of $\mathscr G$ is either finite or locally finite abelian with  elements having order invertible in $R/J(R)$.  Then the primitive ideals of $R\mathscr G$ are precisely the annihilators of modules induced from simple modules of isotropy group rings.
\end{Cor}
\begin{proof}
By Theorem~\ref{t:left.max} it suffices to show that $RG_u$ is a left max ring for each $u\in \mathscr G\skel 0$. If $G_u$ is finite, then $RG_u$ is Artinian and hence a left max ring. Suppose that $G_u$ is locally finite abelian with each element of order invertible in $R/J(R)$.  In particular, $RG_u$ is commutative and it suffices by Hamsher's theorem to show that $J(RG_u)$ is nilpotent and $RG_u/J(RG_u)$ is von Neumann regular. First note that since $J(R)$ is nilpotent, we have that $J(R)RG_u$ is a nilpotent ideal of $RG_u$ and hence contained $J(RG_u)$.  On the other hand, $RG_u/J(R)RG_u\cong (R/J(R))G_u$ is von Neumann regular by~\cite{villamayor} and the hypotheses, and hence has a zero radical.  Thus $J(RG_u)=J(R)RG_u$, and hence is nilpotent, and $RG_u/J(RG_u)$ is von Neumann regular. Therefore,  $RG_u$ is a left max ring.
\end{proof}

If $R$ is  commutative Artinian of characteristic zero, then the order of any element of a locally finite group is invertible in $R/J(R)$ (which is a product of fields of characteristic zero) and so Corollary~\ref{c:main} applies if each isotropy group is either finite or locally finite abelian.

\subsection*{Acknowledgments}
The author would like to thank Enrique Pardo for pointing out Villamayor's result~\cite{villamayor} and that regular rings have zero Jacobson radical, leading to the current formulation of Corollary~\ref{c:main}, which improves on a previous version.


\end{document}